\documentclass[12pt]{amsart}
\usepackage{amssymb,amsmath,amsthm}
\numberwithin{equation}{section}

\newtheorem{thm}{Theorem}[section]
\newtheorem{lemma}[thm]{Lemma}
\newtheorem{remark}[thm]{Remark}

\newtheorem{definition}[thm]{Definition}
\newtheorem{col}[thm]{Corollary}

\begin{document}

\title[Convergence of ESD of Quaternion Sample Covariance Matrices] {Convergence of Empirical Spectral Distributions of Large Dimensional Quaternion Sample Covariance Matrices }

\author{HUIQIN LI,\ \ ZHIDONG BAI, \ \ JIANG HU }
\thanks{Z. D. Bai was partially supported by CNSF  11171057, the Fundamental Research Funds for the Central Universities, PCSIRT, and the NUS Grant R-155-000-141-112; J. Hu was partially supported by a grant CNSF 11301063.
}

\address{KLASMOE and School of Mathematics \& Statistics, Northeast Normal University, Changchun, P.R.C., 130024.}
\email{lihq118@nenu.edu.cn}
\address{KLASMOE and School of Mathematics \& Statistics, Northeast Normal University, Changchun, P.R.C., 130024.}
\email{baizd@nenu.edu.cn}
\address{KLASMOE and School of Mathematics \& Statistics, Northeast Normal University, Changchun, P.R.C., 130024.}
\email{huj156@nenu.edu.cn}

\subjclass{Primary 15B52, 60F15, 62E20;
Secondary 60F17}

\maketitle
\begin{abstract}
In this paper  we establish the limit of the empirical spectral distribution  of  quaternion sample covariance matrices. Suppose  $\mathbf X_n = ({x_{jk}^{(n)}})_{p\times n}$ is a quaternion random matrix. For each $n$, the
entries $\{x_{ij}^{(n)}\}$ are independent random quaternion variables with  a common mean $\mu$ and variance $\sigma^2>0$. It is shown that the empirical spectral distribution  of the quaternion sample covariance matrix $\mathbf S_n=n^{-1}\mathbf X_n\mathbf X_n^*$ converges to the M-P law as $p\to\infty$, $n\to\infty$ and $p/n\to y\in(0,+\infty)$.

{\bf Keywords: } Quaternion matrices, Sample covariance matrix, LSD.

\end{abstract}
\section{Introduction. }

In 1843, Hamilton invented the hyper-complex number of rank 4, to which he gave the name quaternion (see \cite{kuipers1999quaternions}). In fact, research on the quaternion matrices can trace back to 1936 \cite{wolf1936similarity}. After a long blank period, people gradually discover that quaternions and quaternion matrices play  important roles in quantum physics, robot technology and artificial satellite attitude control, and so on, see  \cite{adler1995quaternionic,finkelstein1962foundations}. Thus, studies on quaternions attract considerable attention in resent years, see \cite{zhang1994numerical,zhang1995numerical}, among others.

In addition,   wide application of computer science has increased a thousand folds in terms of computing speed and storage capability in the past ten years. Thus, we need a new theory to analyze huge data sets with  high dimensions. Luckily, the theory of random matrices (RMT) might be a possible one for dealing with these problems. In probability theory and mathematical physics, a random matrix is a matrix-valued random variable. And the sample covariance matrix is one of the most important  random matrices in RMT, which can be traced back to Wishart (1928) \cite{wishart1928generalised}. In 1967,  Mar${\rm \breve{c}}$enko and Pastur  proved that the empirical spectral distribution (ESD) of  large  dimensional complex sample covariance matrices tends to the Mar$\breve{c}$enko-Pastur (M-P) law.
Since then, a lot of  successive studies about large dimensional complex (or real) sample covariance matrix were investigated.  Here the readers are referred to three books \cite{anderson2010introduction,bai2010spectral,mehta2004random} for more details.

In this paper, we  prove that the ESD of  the quaternion sample covariance matrix still converges to the M-P law.
%
First of all, we introduce some notations which will be  used in the paper. Let $A$ be a $p \times p$ Hermitian matrix and denote its eigenvalues  by ${s_j}, j = 1,2, \cdots, p$. The ESD of $A$ is defined by
$${F^A}\left(x\right) =\frac{1}{p}\sum\limits_{j = 1}^p {I\left({s _j} \le x\right)},$$ where ${I\left(D\right)}$ is the indicator function of an event ${D}$.
And the Stieltjes transform of ${F^A}\left(x\right)$ is given by
$$m\left(z\right)=\int_{-\infty}^{+\infty}\frac{1}{x-z}d{F^A}\left(x\right),$$
where $z=u+\upsilon i\in\mathbb{C}^+$.
Let $g\left(x\right)$ and $m_g\left(x\right)$ denote the density function and the Stieltjes transform of  M-P law, that are
\begin{align}\label{des}
g\left(x\right)= \left\{ {\begin{array}{*{20}{c}}
{\frac{1}{{2\pi xy{\sigma ^2}}}\sqrt {\left(b - x\right)\left(x - a\right)} ,}& a \le x \le b;\\
{0,} & otherwise,
\end{array}} \right.
\end{align}
and
\begin{equation}\label{eq:5}
m_g\left(z\right)=\frac{{{\sigma ^2}\left(1 - y\right) - z + \sqrt {{{\left(z - {\sigma ^2} - y{\sigma ^2}\right)}^2} - 4y{\sigma ^2}} }}{{2yz{\sigma ^2}}}
\end{equation}
where $a = {\sigma ^2}{\left(1 - \sqrt y \right)^2}$, $b = {\sigma ^2}{\left(1 + \sqrt y \right)^2}$. If $y > 1$, $G\left(x\right)$, the distribution function of M-P law, has a point mass $1 - 1/y$ at the origin. Here, the constant $y$ is the limit of  dimension $p$ to sample size $n$ ratio and ${\sigma ^2}$  is the scale parameter.
%

Next we shall introduce some notations about quaternion. A quaternion can be represented as a $2 \times 2$ matrix
 $$x = a \cdot \mathbf e + b \cdot \mathbf i + c \cdot \mathbf j + d \cdot \mathbf k =\left( {\begin{array}{*{20}{c}}
a+bi &c+di\\
{ - c+di }&{a-bi }
\end{array}} \right)$$
with the real coefficients  $a,b,c$ and $d$. The quaternion unit can be represented as
\begin{align*}
\mathbf {e} = \left( \begin{array}{cc}
1&0\\
0&1\\
\end{array} \right),
\mathbf i = \left( \begin{array}{cc}
i&0\\
0&- i\\
\end{array} \right),
\mathbf j = \left( \begin{array}{cc}
0&1\\
-1&0\\
\end{array}\right),
\mathbf k = \left( \begin{array}{cc}
0&i\\
i&0\\
\end{array}\right),\end{align*}
where $i$ denotes the imaginary unit. Let the real part and the imaginary part of $x$ be   $\Re x=a \cdot \mathbf e$ and   $\Im x=b \cdot \mathbf i +  c \cdot \mathbf j + d \cdot \mathbf k$ respectively. We denote the conjugate of $x$ be
$$\bar x = a \cdot \mathbf e - b \cdot \mathbf i - c \cdot \mathbf j - d \cdot \mathbf k=\left( {\begin{array}{*{20}{c}}
a-bi &-c-di \\
{ c-di }&{a+bi }
\end{array}} \right), $$
 and the norm of $x$ be
 $$\left\| x \right\| = \sqrt {{a^2} + {b^2} + {c^2} + {d^2}}.$$
More details can be found in  \cite{kuipers1999quaternions,zhang1997quaternions}. Thus, any  $n\times n$ quaternion matrix  $\mathbf X$ can be rewritten as a $2n\times 2n$ complex matrix  $\psi\left(\mathbf X\right)$. Therefore,  we can deal with quaternion matrices as complex matrices for convenience.

Now
our main theorem can be described as following:

\begin{thm}\label{th:1}
Let \ $\mathbf X_n = \left({x_{jk}^{\left(n\right)}}\right)$, $j = 1, \cdots,p; k= 1,\cdots,n$. Suppose for each $n$, $\left\{x_{jk}^{\left(n\right)}\right\}$  are independent quaternion random variables  with a common mean $\mu$ and variance $\sigma^2$.
Assume that $y_n=p/n \to y \in \left(0,\infty \right)$ and for any constant $\eta > 0$,
\begin{equation}\label{eq:1}
  \frac{1}{{np}}{\sum\limits_{jk} {{\rm E}\left\| {x_{jk}^{(n)}} \right\|} ^2}I\left(\left\| {x_{jk}^{(n)}} \right\| > \eta \sqrt n \right) \to 0.
\end{equation}
Then, with probability one, the ESD of sample covariance matrix $\mathbf S_n=\frac{1}{n}{\mathbf X_n}{\mathbf X_n^*}$ converges to the M-P law which has density function \eqref{des} and a point mass $1-1/y$ at the
 origin when $y>1$. Here superscript $^*$ stands for the  complex conjugate transpose.
\end{thm}

\begin{remark}
Without loss of generality, in the proof of Theorem \ref{th:1}, we assume that $  \sigma^2=1$. One can see that removing the common mean of the entries of $\mathbf X_n$ does not alter the LSD of sample covariance matrices. In fact, let
\begin{align*}
\mathbf T_n=\frac{1}{n}\left(\mathbf {X_n-{\rm E}X_n}\right)\left(\mathbf {X_n-{\rm E}X_n}\right)^*.
\end{align*}
By Lemma \ref{lemma:3}, we have, for all large $p$,
\begin{align*}
\left\| {{F^{\mathbf S_n}} - {F^{\mathbf T_n}}} \right\| \le \frac{1}{2p}rank\left({\rm E}\mathbf X_n\right)\le \frac1p\to 0.
\end{align*}
Furthermore, we assume that $\mu=0$.
\end{remark}

The paper is organized as follows. In Section 2, the structure of the inverse of a kind of  matrices is established which is the key tool of proving Theorem \ref{th:1}. We prove the main theorem by the Stieltjes transform method in Section 3. And in Section 4, we outline some auxiliary lemmas which can be used in Section 3.

\section{Preliminaries. }

We shall use Lemma 2.5 proved by Yin, Bai and Hu in \cite{yin2013semicircular} to prove our main results in next section. For being self-contained, this lemma is now stated as follows.
\begin{definition} A matrix is called Type-\uppercase\expandafter{\romannumeral1} matrix if it has the following structure:
\[\left( {\begin{array}{*{20}{c}}
{{t_1}}&0&{{a_{12}}}&{{b_{12}}}& \cdots &{{a_{1n}}}&{{b_{1n}}}\\
0&{{t_1}}&{{c_{12}}}&{{d_{12}}}& \cdots &{{c_{1n}}}&{{d_{1n}}}\\
{{d_{12}}}&{ - {b_{12}}}&{{t_2}}&0& \cdots &{{a_{2n}}}&{{b_{2n}}}\\
{ - {c_{12}}}&{{a_{12}}}&0&{{t_2}}& \cdots &{{c_{2n}}}&{{d_{2n}}}\\
 \vdots & \vdots & \vdots & \vdots & \ddots & \vdots & \vdots \\
{{d_{1n}}}&{ - {b_{1n}}}&{{d_{2n}}}&{ - {b_{2n}}}& \cdots &{{t_n}}&0\\
{ - {c_{1n}}}&{{a_{1n}}}&{ - {c_{2n}}}&{{a_{2n}}}& \ldots &0&{{t_n}}
\end{array}} \right).\]
Here all the  entries are  complex.
\end{definition}
\begin{definition}A matrix is called Type-\uppercase\expandafter{\romannumeral2} matrix if it has the following structure:
\begin{footnotesize}
\[\left( {\begin{array}{*{20}{c}}
{{t_1}}&0&{{a_{12}} + {c_{12}}  i}&{{b_{12}} + {d_{12}}  i}& \cdots &{{a_{1n}} + {c_{1n}}  i}&{{b_{1n}} + {d_{1n}}  i}\\
0&{{t_1}}&{ - {{\bar b}_{12}} - {{\bar d}_{12}}  i}&{{{\bar a}_{12}} + {{\bar c}_{12}}  i}& \cdots &{ - {{\bar b}_{1n}} - {{\bar d}_{1n}}  i}&{{{\bar a}_{1n}} + {{\bar c}_{1n}}  i}\\
{{{\bar a}_{12}} + {{\bar c}_{12}}  i}&{ - {b_{12}} - {d_{12}}  i}&{{t_2}}&0& \cdots &{{a_{2n}} + {c_{2n}}  i}&{{b_{2n}} + {d_{2n}}  i}\\
{{{\bar b}_{12}} + {{\bar d}_{12}}  i}&{{a_{12}} + {c_{12}}  i}&0&{{t_2}}& \cdots &{ - {{\bar b}_{2n}} - {{\bar d}_{2n}}  i}&{{{\bar a}_{2n}} + {{\bar c}_{2n}}  i}\\
 \vdots & \vdots & \vdots & \vdots & \ddots & \vdots & \vdots \\
{{{\bar a}_{1n}} + {{\bar c}_{1n}}  i}&{ - {b_{1n}} - {d_{1n}}  i}&{{{\bar a}_{2n}} + {{\bar c}_{2n}}  i}&{ - {b_{2n}} - {d_{2n}} i}& \cdots &{{t_n}}&0\\
{{{\bar b}_{1n}} + {{\bar d}_{1n}}  i}&{{a_{1n}} + {c_{1n}}  i}&{{{\bar b}_{2n}} + {{\bar d}_{2n}}  i}&{{a_{2n}} + {c_{2n}}  i}& \ldots &0&{{t_n}}
\end{array}} \right).\]
\end{footnotesize}
Here $i=\sqrt{-1}$ denotes the usual imaginary unit and all the other variables are  complex numbers.
\end{definition}
\begin{definition}A matrix is called Type-\uppercase\expandafter{\romannumeral3} matrix if it has the following structure:
\[\left( {\begin{array}{*{20}{c}}
{{t_1}}&0&{{a_{12}} }&{{b_{12}} }& \cdots &{{a_{1n}}}&{{b_{1n}} }\\
0&{{t_1}}&{ - {{\bar b}_{12}} }&{{{\bar a}_{12}} }& \cdots &{ - {{\bar b}_{1n}} }&{{{\bar a}_{1n}}}\\
{{{\bar a}_{12}} }&{ - {b_{12}} }&{{t_2}}&0& \cdots &{{a_{2n}} }&{{b_{2n}} }\\
{{{\bar b}_{12}} }&{{a_{12}}}&0&{{t_2}}& \cdots &{ - {{\bar b}_{2n}} }&{{{\bar a}_{2n}} }\\
 \vdots & \vdots & \vdots & \vdots & \ddots & \vdots & \vdots \\
{{{\bar a}_{1n}}}&{ - {b_{1n}} }&{{{\bar a}_{2n}}}&{ - {b_{2n}} }& \cdots &{{t_n}}&0\\
{{{\bar b}_{1n}} }&{{a_{1n}} }&{{{\bar b}_{2n}} }&{{a_{2n}} }& \ldots &0&{{t_n}}
\end{array}} \right).\]
Here all the  entries are  complex.
\begin{lemma}
For all $n\geq1$, if a complex  matrix  $\Omega_n$ is  invertible and of Type-$I\!I$, then $\Omega_n^{-1}$ is a Type-\uppercase\expandafter{\romannumeral1} matrix.
\end{lemma}
\end{definition}
The following corollary is immediate.
\begin{col}\label{col:1}
For all $n\geq1$, if a complex  matrix  $\Omega_n$ is   invertible and of Type-\uppercase\expandafter{\romannumeral3}, then $\Omega_n^{-1}$ is a Type-\uppercase\expandafter{\romannumeral1} matrix.
\end{col}

\section{Proof of Theorem \ref{th:1}}

In this section, we complete the proof by the following two steps. The first one is to truncate, centralize and rescale the random variables $\{x_{ij}^{\left(n\right)}\}$, then we may assume the additional conditions which will be given in Remark \ref{le:3}. In the second part of this section we give the proof of  the Theorem \ref{th:1} by Stieltjes transform.

\subsection{Truncation, Centralization and Rescaling}

\subsubsection{Truncation}

Note that condition (\ref{eq:1}) is equivalent to: for any $\eta > 0$,
\begin{equation}\label{eq:2}
    \mathop {\lim }\limits_{n \to \infty } \frac{1}{{{\eta ^2}np}}{\sum\limits_{jk} {{\rm E}\left\| {x_{jk}^{\left(n\right)}} \right\|} ^2}I\left(\left\| {x_{jk}^{\left(n\right)}} \right\| > \eta \sqrt n \right) = 0.
   \end{equation}
Thus, one can select a sequence ${\eta _n} \downarrow 0$ such that (\ref{eq:2}) remains true when $\eta$ is replaced by $\eta_n$.

\begin{lemma}\label{le:1}
Suppose that the assumptions of Theorem \ref{th:1} hold. Truncate the variables  $x_{jk}^{\left(n\right)}$ at ${\eta _n}\sqrt n$, and denote the resulting variables by $\widehat x_{jk}^{\left(n\right)}$. Write $$\widehat {x}_{jk}^{\left(n\right)}=x_{jk}^{\left(n\right)}I\left(\left\|x_{jk}^{\left(n\right)}\right\|\leq {\eta _n}\sqrt n\right), \ \widehat {\mathbf X}_n=\left(\widehat {x}_{jk}^{\left(n\right)}\right)~~\mbox{and}~~ \ \widehat{\mathbf S}_n=\frac{1}{n}{\widehat {\mathbf X}_n}{\widehat {\mathbf X}_n^*}.$$
Then, with probability 1,
  \begin{equation}
  \left\| {{F^{{\mathbf S_n}}} - {F^{\widehat {\mathbf S}_n}}} \right\|=\sup_x\left|{{F^{{\mathbf S_n}}(x)} - {F^{\widehat{\mathbf S}_n}}(x)}\right| \to 0.\nonumber
\end{equation}
\end{lemma}
\begin{proof} By using Lemma \ref{lemma:3}, one has
\begin{align}\label{eq:3}
  \left\| {{F^{{\mathbf S_n}}} - {F^{\widehat{\mathbf S}_n}}} \right\| &\le \frac{1}{{2p}}{\rm rank}\left( {\frac{1}{\sqrt n}{\mathbf X_n} - {\frac{1}{\sqrt n}{\widehat{\mathbf X}}_n}} \right)\notag\\
  & \leq \frac{1}{{2p}}\sum\limits_{jk} {I\left(\left\| {x_{jk}^{(n)}} \right\| >{\eta _n}\sqrt n \right)}.
\end{align}
Taking condition (\ref{eq:2}) into consideration, we get
\begin{align*}
 &{\rm E}\left(\frac{1}{2p}\sum\limits_{jk} {I\left(\left\| {x_{jk}^{(n)}} \right\| > {\eta _n}\sqrt n \right)} \right) \\ \le &\frac{1}{{2\eta _n^2{np}}}{\sum\limits_{jk} {{\rm E}\left\| {x_{jk}^{(n)}} \right\|} ^2}I
 \left(\left\| {x_{jk}^{(n)}} \right\| > {\eta _n}\sqrt n \right) =o\left(1\right)
 \end{align*}
 and
\begin{align*}
& {\rm Var}\left(\frac{1}{2p}\sum\limits_{jk} {I\left(\left\| {x_{jk}^{(n)}} \right\| > {\eta _n}\sqrt n \right)} \right) \\
\le& \frac{1}{{4\eta _n^2{p^2}n}}{\sum\limits_{jk} {{\rm E}\left\| {x_{jk}^{\left(n\right)}} \right\|} ^2}I\left(\left\| {x_{jk}^{\left(n\right)}} \right\| > {\eta _n}\sqrt n \right) =o\left(\frac{1}{p}\right).
\end{align*}
 Then by Bernstein's inequality (see Lemma \ref{lemma:10}), for all small $\varepsilon > 0 $ and large $n$, we obtain
 \begin{align*}
  {\rm  P}\left(\frac{1}{2p}\sum\limits_{jk } {I\left(\left\| {x_{jk}^{(n)}} \right\| > {\eta _n}\sqrt n \right)}  \ge \varepsilon \right) \le 2{e^{ - \varepsilon p/2}}
 \end{align*}
 which implies that
  \begin{equation}\label{eq:4}
 \sum_{} {\rm  P}\left(\frac{1}{2p}\sum\limits_{jk} {I\left(\left\| {x_{jk}^{(n)}} \right\| > {\eta _n}\sqrt n \right)} \ge \varepsilon \right) <\infty.
 \end{equation}
  Together with (\ref{eq:3}), (\ref{eq:4}) and Borel-Cantelli lemma, we obtain
  \begin{equation}
  \left\| {{F^{{\mathbf S_n}}} - {F^{\widehat{\mathbf S}_n}}} \right\| \to 0, \ \mbox{a.s..}\nonumber
\end{equation}
This completes the proof of the lemma.
\end{proof}

\subsubsection{Centralization}

\begin{lemma}\label{lemma:11}
Suppose that the assumptions of Lemma \ref{le:1}  hold.
Denote $$\widetilde {x}_{jk}^{\left(n\right)}=\widehat {x}_{jk}^{\left(n\right)}-{\rm E}\widehat {x}_{jk}^{\left(n\right)}, \ \widetilde {\mathbf X}_n=\left(\widetilde {x}_{jk}^{\left(n\right)}\right)~~\mbox{and}~~ \ \widetilde{\mathbf S}_n=\frac{1}{n}{\widetilde {\mathbf X}_n}{\widetilde {\mathbf X}_n^*}.$$
Then, we obtain
\begin{displaymath}
L\left(F^{\widehat{\mathbf S}_n}, F^{\widetilde{\mathbf S}_n}\right)=o\left(1\right),
\end{displaymath}
 where $L\left(\cdot,\cdot \right)$  denotes the L\'{e}vy distance.
\end{lemma}
\begin{proof} Using Lemma \ref{lemma:4} and condition (\ref{eq:2}), we have
\begin{align}\label{al:16}
&L^4\left(F^{\widehat{\mathbf S}_n},F^{\widetilde{\mathbf S}_n}\right)\notag\\
\le& \frac{1}{2p^2}\left({\rm tr}\left(\widehat{\mathbf S}_n+\widetilde{\mathbf S}_n\right)\right)\left({\rm tr}\left(\frac{1}{\sqrt n}\widehat{\mathbf X}_n-\frac{1}{\sqrt n}\widetilde{\mathbf X}_n\right)\left(\frac{1}{\sqrt n}\widehat{\mathbf X}_n-\frac{1}{\sqrt n}\widetilde{\mathbf X}_n\right)^*\right)\notag\\
=&\frac{1}{2n^2p^2}\left(\sum_{jk}^{}\left(\left\|\widehat x_{jk}^{\left(n\right)}\right\|^2+\left\|\widehat x_{jk}^{\left(n\right)}-{\rm E}\widehat x_{jk}^{\left(n\right)}\right\|^2\right)\right)\left(\sum_{jk}^{}\left\|{\rm E}\widehat x_{jk}^{\left(n\right)}\right\|^2\right)\notag\\
=&\left(\frac{1}{np}\sum_{jk}^{}\left(\left\|\widehat x_{jk}^{\left(n\right)}\right\|^2+\left\|\widehat x_{jk}^{\left(n\right)}-{\rm E}\widehat x_{jk}^{\left(n\right)}\right\|^2\right)\right)\left(\frac{1}{2np}\sum_{jk}^{}\left\|{\rm E}\widehat x_{jk}^{\left(n\right)}\right\|^2\right).
\end{align}
Applying Lemma \ref{lemma:5}, one has
\begin{align*}
&{\rm E}\left|\frac{1}{np}\sum_{jk}^{}\left(\left\|\widehat x_{jk}^{\left(n\right)}\right\|^2-{\rm E}\left\|\widehat x_{jk}^{\left(n\right)}\right\|^2\right)\right|^4\\
\le&\frac{C}{n^4p^4}\sum_{jk}{\rm E}\left\|\widehat x_{jk}^{\left(n\right)}\right\|^8+\left(\sum_{j,k}{\rm E}\left\|\widehat x_{jk}^{\left(n\right)}\right\|^4\right)^2\\
\le&Cn^{-2}\left(\eta_n^6n^{-1}y_n^{-3}+\eta_n^4y_n^{-2}\right).
\end{align*}
By Borel-Cantelli lemma, we have
\begin{align*}
\frac{1}{np}\sum_{jk}^{}\left(\left\|\widehat x_{jk}^{\left(n\right)}\right\|^2-{\rm E}\left\|\widehat x_{jk}^{\left(n\right)}\right\|^2\right)\to 0 \ {\rm a.s..}
\end{align*}
Similarly, we can obtain
\begin{align}\label{al:17}
\frac{1}{np}\sum_{jk}^{}\left(\left\|\widehat x_{jk}^{\left(n\right)}-{\rm E}\widehat x_{jk}^{\left(n\right)}\right\|^2-{\rm E}\left\|\widehat x_{jk}^{\left(n\right)}-{\rm E}\widehat x_{jk}^{\left(n\right)}\right\|^2\right)\to 0 \ {\rm a.s..}
\end{align}
Thus, by (\ref{al:16}), for all large $n$,
\begin{align*}
&L^4\left(F^{\widehat{\mathbf S}_n},F^{\widetilde{\mathbf S}_n}\right)\\
\le&\left(\frac{1}{np}\sum_{jk}^{}\left({\rm E}\left\|\widehat x_{jk}^{\left(n\right)}\right\|^2+{\rm E}\left\|\widehat x_{jk}^{\left(n\right)}-{\rm E}\widehat x_{jk}^{\left(n\right)}\right\|^2\right)+o_{a.s.}(1)\right)\left(\frac{1}{2np}\sum_{jk}^{}\left\|{\rm E}\widehat x_{jk}^{\left(n\right)}\right\|^2\right)\\
\le&\frac{C}{np}\sum_{jk}^{}\left\|{\rm E}\widehat x_{jk}^{\left(n\right)}\right\|^2\\
\le&\frac{C}{np}\sum_{jk}^{}{\rm E}\left\|x_{jk}^{\left(n\right)}\right\|^2I\left(\left\|x_{jk}^{\left(n\right)}\right\|\ge\eta_n\sqrt n\right)\to 0.
\end{align*}
The proof of the lemma is complete.
\end{proof}

\subsubsection{Rescaling}
Define
$$\widetilde \sigma_{jk}^2={\rm E}\left\|\widetilde x_{jk}^{\left(n\right)}\right\|^2, \xi_{jk}= \left\{ {\begin{array}{*{20}{c}}
\zeta_{jk} ,& \widetilde \sigma_{jk}^2 < 1/2\\
\widetilde {x}_{jk}^{\left(n\right)} , & \widetilde\sigma_{jk}^2 \ge 1/2
\end{array}} \right., \Lambda= \frac{1}{\sqrt n}\left(\xi_{jk}\right), \sigma_{jk}^2={\rm E}\left\| \xi_{jk}\right\|^2,$$
where $\zeta_{jk}$ is a bounded quaternion random variable with ${\rm E}\zeta_{jk}=0$, ${\rm Var}\zeta_{jk}=1$.
\begin{lemma}\label{le:2}
Write
\begin{gather*}
\breve {x}_{jk}^{\left(n\right)}={ \sigma_{jk}^{-1} {\xi}_{jk}}, \ \breve {\mathbf X}_n=\left(\breve {x}_{jk}^{\left(n\right)}\right), \ {\rm and} \ \breve{\mathbf S}_n=\frac{1}{n}{\breve {\mathbf X}_n}{\breve {\mathbf X}_n^*}.
\end{gather*}
Under the conditions assumed in Lemma \ref{lemma:11}, we have
\begin{displaymath}
L\left(F^{\breve{\mathbf S}_n}, F^{\widetilde{\mathbf S}_n}\right)=o\left(1\right).
\end{displaymath}
\end{lemma}
\begin{proof} a): Our first goal is to show  that  $$L\left(F^{\widetilde{\mathbf S}_n} , F^{\Lambda\Lambda^*}\right) \to 0, \ \mbox{a.s..}$$
	 Let ${\mathcal E}_n$ be the set of pairs $\left(j,k\right)$ : $\widetilde\sigma _{jk}^2 < \frac{1}{2}$ and $N_n=\sum \limits_{\left(j,k\right)\in {\mathcal E}_n} I\left(\widetilde\sigma _{jk}^2 <1/2\right) $. Because $\frac{1}{np}\sum \limits_{jk} \widetilde \sigma _{jk}^2 \to 1 $, we conclude that $N_n=o\left(np\right)$. Owing to Lemma \ref{lemma:4} and (\ref{al:17}), we get:
  \begin{align}\label{al:2}
&{L^4}\left(F^{\widetilde {\mathbf S}_n} , F^{\Lambda\Lambda^*}\right)\notag\\
\le &\frac{1}{2p^2}\left({\rm tr}\left(\widetilde{\mathbf S}_n+\Lambda\Lambda\right)\right)\left({\rm tr}\left(\frac{1}{\sqrt n}\widetilde{\mathbf X}_n-\Lambda\right)\left(\frac{1}{\sqrt n}\widetilde{\mathbf X}_n-\Lambda\right)^*\right)\notag\\
=&\frac{1}{2n^2p^2}\left(\sum_{jk}^{}\left(\left\|\widetilde x_{jk}^{(n)}\right\|^2+\left\|\xi_{jk}\right\|^2\right)\right)\left(\sum_{jk}^{}\left\|\xi_{jk}-\widetilde x_{jk}^{(n)}\right\|^2\right)\notag\\
=&\frac{1}{2n^2p^2}\left(\sum_{jk}^{}{\rm E}\left(\left\|\widetilde x_{jk}^{(n)}\right\|^2+\left\|\xi_{jk}\right\|^2\right)+o_{\rm a.s.}(1)\right)\left(\sum_{jk}^{}\left\|\xi_{jk}-\widetilde x_{jk}^{(n)}\right\|^2\right)\notag\\
\le&\frac{C}{np}\sum_{jk}^{}\left\|\xi_{jk}-\widetilde x_{jk}^{(n)}\right\|^2
:=\frac{C}{np}\sum\limits_{h=1}^K u_h
\end{align}
where $K=N_n$ and $u_h=\left\|\xi_{jk}-\widetilde x_{jk}^{(n)}\right\|^2$. Then, using the fact that for all $l\geq 1$, $l!\geq\left(l/3\right)^l$, we have
\begin{align*}
{\rm E}\left(\frac{1}{np}\sum\limits_{h=1}^K u_h\right)^m =&\frac{1}{ n^mp^{m}}\sum\limits_{m_1+\ldots+m_K=m}\frac{m!}{{m_1}!\ldots{m_K}!}{\rm E}u_1^{m_1}\ldots {\rm E}u_K^{m_K} \\
&\leq \frac{1}{ n^{m}p^{m}}\sum\limits_{l=1}^m \sum_{\substack{m_1+\ldots +m_l=m \\ m_t\geq 1}}\frac{m!}{l!m_1! \ldots m_l!}\prod_{t=1}^l \left(\sum_{h=1}^K {\rm E}u_h^{m_t}\right)
 \\
&\leq C \sum_{l=1}^m  n^{-m} p^{-m}l^m\left(l!\right)^{-1}\left(2\eta_n^2 n\right)^{m-l}2^l K^l \\
&\leq C\sum_{l=1}^m\left(\frac{6K}{np}\right)^l\left(\frac{2\eta_n^2l}{p}\right)^{m-l}
\leq C\left(\frac{6K}{np}+\frac{2\eta_n^2 m}{p}\right)^m.
\end{align*}
By selecting m=[log p] that implies \ $\frac{2\eta_n^2m}{p}\to 0$, and noticing $\frac{6K}{np} \to 0$, we have for any fixed $t>0$,
\begin{align*}
{\rm E}\left(\frac{1}{np}\sum\limits_{h=1}^K u_h\right)^m \le o\left(p^{-t}\right).
\end{align*}
From the inequality above with $t=2$ and (\ref{al:2}), it follows that
$$L\left(F^{\widetilde{\mathbf S}_n} , F^{\Lambda\Lambda^*}\right) \to 0, \ \mbox{a.s..}$$
b): Our next goal is to show that $$L\left(F^{\breve{\mathbf S}_n}, F^{\Lambda\Lambda^*}\right) \to 0, \ \mbox{a.s..}$$
Applying Lemma \ref{lemma:4}, we conclude that:
  \begin{align}
&{L^4}\left(F^{\breve {\mathbf S}_n} , F^{\Lambda\Lambda^*}\right)\notag\\
\le &\frac{1}{2p^2}\left({\rm tr}\left(\breve{\mathbf S}_n+\Lambda\Lambda\right)\right)\left({\rm tr}\left(\frac{1}{\sqrt n}\breve{\mathbf X}_n-\Lambda\right)\left(\frac{1}{\sqrt n}\breve{\mathbf X}_n-\Lambda\right)^*\right)\notag\\
=&\frac{1}{2n^2p^2}\left(\sum_{jk}^{}\left(\left\|\breve x_{jk}^{(n)}\right\|^2+\left\|\xi_{jk}\right\|^2\right)\right)\left(\sum_{jk}^{}\left\|\xi_{jk}-\breve x_{jk}^{(n)}\right\|^2\right)\notag\\
=&\frac{1}{2n^2p^2}\left(\sum_{jk}^{}\left(1+\sigma_{jk}^{-2}\right){\rm E}\left\|\xi_{jk}\right\|^2+o_{\rm a.s.}(1)\right)\left(\sum_{jk}^{}\left(1-\sigma_{jk}^{-1}\right)^2\left\|\xi_{jk}\right\|^2\right)\notag\\
\le&\frac{C}{np}\sum_{jk}^{}\left(1-\sigma_{jk}^{-1}\right)^2\left\|\xi_{jk}\right\|^2\notag.
\end{align}
Using the fact
 \begin{align*}
&{\rm E}\left(\frac{C}{np}\sum_{jk}^{}\left(1-\sigma_{jk}^{-1}\right)^2\left\|\xi_{jk}\right\|^2\right)
= \frac{C}{np}\sum_{jk} \left(1- \sigma_{jk}\right)^2
\leq \frac{C}{np}\sum_{jk} \left(1- \sigma_{jk}^2\right) \\
\leq & \frac{C\eta_n^2}{{{\eta_n^2np}}}\sum_{(j,k) \not \in {\mathcal E}_n } \left[{\rm E}\left\|x_{jk}^{\left(n\right)}\right\|^2 I\left(\left\|x_{jk}^{\left(n\right)}\right\|\geq \eta_n \sqrt n\right)+{\rm E}^2\left\|x_{jk}^{\left(n\right)}\right\|I\left(\left\|x_{jk}^{\left(n\right)}\right\|\geq \eta_n \sqrt n\right)\right]\\
 \to& 0
\end{align*}
and by Lemma \ref{lemma:5}, we get
 \begin{align*}
&{\rm E } \left |\frac{C}{np}\sum_{jk}^{}\left(1-\sigma_{jk}^{-1}\right)^2\left(\left\|\xi_{jk}\right\|^2-{\rm E}\left\|\xi_{jk}\right\|^2\right)\right |^4 \\
\leq & \frac{C}{n^4p^4}\left [\sum_{j, k}{\rm E}\left\|x_{jk}^{\left(n\right)}\right\|^8 I\left(\left\|x_{jk}^{\left(n\right)}\right\|\leq \eta_n \sqrt n\right)+\left( \sum_{j, k} {\rm E}\left\|x_{jk}^{\left(n\right)}\right\|^4 I\left(\left\|x_{jk}^{\left(n\right)}\right\|\leq \eta_n \sqrt n\right)\right )^2 \right] \\
\leq & Cn^{-2}\left[n^{-1}\eta_n^6y_n^{-3}+\eta_n^4y_n^{-2}\right]
\end{align*}
which is summable. Together with Borel-Cantelli lemma, it follows that
$$L\left(F^{\breve{\mathbf S}_n} , F^{\Lambda\Lambda^*}\right) \to 0, \ \mbox{a.s..}$$
c): Finally, from a) and b), we can easily get the lemma.
\end{proof}

Combining the results of Lemma \ref{le:1}, Lemma \ref{lemma:11}, and Lemma \ref{le:2}, we have the following remarks.
\begin{remark}\label{le:3}
Under the conditions assumed in Theorem \ref{th:1}, we can further assume that \\
1) $\left\|x_{jk}\right\|\le\eta_n\sqrt n$,\\
2)  ${\rm E}\left(x_{jk}\right)=0 $ and $ {\rm Var}\left(x_{jk}\right)=1$.
\end{remark}
\begin{remark}
For brevity, we shall drop the superscript (n) from the variables. Also the truncated and renormalized variables are still denoted by $x_{jk}$.
\end{remark}

\subsection {Completion of the proof}
Denote
$$m_n\left(z\right)=\frac{1}{2p}{\rm tr}\left(\mathbf S_n-z\mathbf I_{2p}\right)^{-1},$$
where $z=u+\upsilon i\in\mathbb{C}^+$.
\subsubsection {Random part}\label{se:1}
Firstly, we should show that
\begin{equation}\label{eq:6}{m_n}\left(z\right) - {\rm E}{m_n}\left(z\right) \to 0, \ \mbox{a.s..}\end{equation}
Let $\pi_j$ denote the $j$-th column of \ $\mathbf X_n$, $\mathbf S_n^k = {\mathbf S_n} - \frac{1}{n}{\pi _k}\pi _k^*$ and  ${{\rm E}_k}\left( \cdot \right)$ denote the conditional expectation given $\left\{ {{\pi _{k + 1}},{\pi _{k + 2}}, \cdots ,{\pi_{2n}}} \right\}$. Then
\begin{align*}
{m_n}\left(z\right) - {\rm E}{m_n}\left(z\right) = &\frac{1}{{2p}}\sum\limits_{k = 1}^{2n}\left[ {{{\rm  E}_{k - 1}}{\rm tr}{{\left({\mathbf S_n} - z{\mathbf I_{2p}}\right)}^{ - 1}} - } {{\rm E}_k}{\rm tr}{\left({\mathbf S_n} - z{\mathbf I_{2p}}\right)^{ - 1}}\right]\\
 = &\frac{1}{{2p}}\sum\limits_{k = 1}^{2n} {{\gamma _k}} ,
\end{align*}
where
\begin{align*}
{\gamma _k} = &{{{\rm  E}_{k - 1}}{\rm tr}{{\left({\mathbf S_n} - z{\mathbf I_{2p}}\right)}^{ - 1}} - } {{\rm E}_k}{\rm tr}{\left({\mathbf S_n} - z{\mathbf I_{2p}}\right)^{ - 1}}\\
=&\left({{\rm E}_{k - 1}} - {{\rm E}_k}\right)\left[{\rm tr}{\left({\mathbf S_{n}} - z{\mathbf I_{2p}}\right)^{ - 1}} - {\rm tr}{\left(\mathbf S_n^k - z{\mathbf I_{2p}}\right)^{ - 1}}\right].
\end{align*}

\begin{description}
\item [1] When $k = 2t - 1 \left(t =1,2, \cdots ,n\right)$, because of the property  of quaternion matrices, we can obtain
\begin{align*}
{\gamma _k} =&{{{\rm  E}_{k - 1}}{\rm tr}{{\left({\mathbf S_n} - z{\mathbf I_{2p}}\right)}^{ - 1}} - } {{\rm E}_k}{\rm tr}{\left({\mathbf S_n} - z{\mathbf I_{2p}}\right)^{ - 1}}
= 0.
\end{align*}
\item [2] When $k = 2t \left(t =0,1, \cdots ,n\right)$, together with  the formula \[{\left(A + \alpha {\beta ^ * }\right)^{ - 1}} = {A^{ - 1}} - \frac{{{A^{ - 1}}\alpha {\beta ^ * }{A^{ - 1}}}}{{1 + {\beta ^ * }{A^{ - 1}}\alpha }},\]
    we obtain
    \begin{align*}
    {\gamma _k} =&\left({{\rm E}_{k - 1}} - {{\rm E}_k}\right)\left[{\rm tr}{\left({\mathbf S_{n}} - z{\mathbf I_{2p}}\right)^{ - 1}} - {\rm tr}{\left(\mathbf S_n^k - z{\mathbf I_{2p}}\right)^{ - 1}}\right]\\
    = &\left({{\rm E}_{k - 1}} - {{\rm E}_k}\right)\frac{\frac{1}{n}{\pi_{k}^ * {{\left(\mathbf S_n^k - z{\mathbf I_{2p}}\right)}^{ - 2}}{\pi_{k}}}}{{1 +\frac{1}{n} \pi_{k}^ * {{\left(\mathbf S_n^k - z{\mathbf I_{2p}}\right)}^{ - 1}}{\pi_{k}}}}.
    \end{align*}
Since
\begin{align*}
&\left| \frac{\frac{1}{n}{\pi_{k}^ * {{\left(\mathbf S_n^k - z{\mathbf I_{2p}}\right)}^{ - 2}}{\pi_{k}}}}{{1 +\frac{1}{n} \pi_{k}^ * {{\left(\mathbf S_n^k - z{\mathbf I_{2p}}\right)}^{ - 1}}{\pi_{k}}}} \right|\\
\le& \frac{\frac{1}{n}{\pi_{k}^ * {{\left({{\left(\mathbf S_n^k - u{\mathbf I_{2p}}\right)}^2} + {\upsilon^2}{\mathbf I_{2p}}\right)}^{ - 1}}{\pi_{k}}}}{{\Im \left(1 +\frac{1}{n} \pi_{k}^ * {{\left(\mathbf S_n^k - z{\mathbf I_{2p}}\right)}^{ - 1}}{\pi_{k}}\right)}}\\
=& \frac{1}{\upsilon},
\end{align*}
we can easily get
$$\left|\gamma_k\right| \le \frac{2}{\upsilon} .$$
\end{description}
Using Lemma \ref{lemma:8}, it follows that
 $${\rm E}{\left| {{m_{n}}\left(z\right) - {\rm E}{m_{n}}\left(z\right)} \right|^4} \le \frac{{{K_4}}}{{{{\left(2p\right)}^4}}}{\rm E}{\left( {\sum\limits_{k = 1}^{2n} {{{\left| {{\gamma _k}} \right|}^2}} } \right)^2} \le \frac{{{4 K_4}{n^2}}}{{{p^4}{v^4}}} = O\left({n^{ - 2}}\right).$$
Combining with Borel-Cantelli lemma and Chebyshev inequality, we complete the proof  that  $${m_n}\left(z\right) - {\rm E}{m_n}\left(z\right) \to 0,\ \mbox{a.s..}$$

\subsubsection{ Mean convergence}\label{se:2}
When $\sigma^2=1$, (\ref{eq:5}) turns into
\begin{align}\label{al:3}
m\left(z\right)=\frac{{1 - y- z + \sqrt {{{\left(1-z-y\right)}^2} - 4yz} }}{{2yz}}.
\end{align}
Next we will devote to prove that
$${\rm E}{m_n}\left(z\right) \to m\left(z\right).$$
Applying Lemma \ref{lemma:9}, one has\[{m_n}\left(z\right) = \frac{1}{{2p}}\sum\limits_{k = 1}^p {{\rm tr}\left(\frac{1}{n}\boldsymbol{\phi}_k^{\prime}\bar{\boldsymbol{\phi}}_k - z{\mathbf I_2} - \frac{1}{{{n^2}}}\boldsymbol{\phi}_k^{\prime}{\mathbf X}_{nk}^*{{\left(\frac{1}{n}{\mathbf X_{nk}}{\mathbf X_{nk}^*} - z{\mathbf I_{2p - 2}}\right)}^{-1}}{\mathbf X_{n k}}\bar{\boldsymbol{\phi}}_k\right )}^{ - 1}\]
where \ ${\mathbf X_{nk}}$ is the matrix resulting from deleting the $k$-th quaternion row of \ $\mathbf X_n$, and $\boldsymbol{\phi}_k^{\prime}$  is the vector obtained from the $k$-th quaternion row of \ $\mathbf X_n$. Notice that $\boldsymbol{\phi}_k^{\prime}$ can be represent as a $2\times 2n$ matrix.
Set \begin{align*}
{\boldsymbol \varepsilon _k} &=\frac{1}{n}\boldsymbol{\phi}_k^{\prime}\bar{\boldsymbol{\phi}}_k - z{\mathbf I_2} - \frac{1}{{{n^2}}}\boldsymbol{\phi}_k^{\prime}{\mathbf X}_{nk}^*{{\left(\frac{1}{n}{\mathbf X_{nk}}{\mathbf X_{nk}^*} - z{\mathbf I_{2p - 2}}\right)}^{-1}}{\mathbf X_{nk}}\bar{\boldsymbol{\phi}}_k \\
&- \left(1 - z - {y_n} - {y_n}z{\rm E}{m_n}\left(z\right)\right){\mathbf I_2}
\end{align*}
and
\begin{align}\label{al:11}
{\delta _n} = & - \frac{1}{2p\left({1 - z - {y_n} - {y_n}z{\rm E}{m_n}\left(z\right)}\right)}\notag\\
\times&\sum\limits_{k = 1}^p {\rm Etr}\left\{{\boldsymbol\varepsilon _k}{\left(\left(1 - z - {y_n} - {y_n}z{\rm E}{m_n}(z)\right){\mathbf I_2} + {\boldsymbol\varepsilon _k}\right)^{ - 1}}\right\}\end{align}
where \ $y_n=p/n$.
This implies that $${\rm E}{m_n}\left(z\right) = \frac{1}{{1 - z - {y_n} - {y_n}z{\rm E}{m_n}\left(z\right)}} + {\delta _n}.$$
Solving ${\rm E}{m_n}\left(z\right)$ from the equation above, we get
\begin{align*}
{\rm E}m_n\left(z\right)=\frac{1}{2y_nz}\left(1-z-y_n+y_nz\delta_n\pm \sqrt{\left(1-z-y_n-y_nz\delta_n\right)^2-4y_nz}\right).
\end{align*}
As proved in the equation (3.17) of Bai \cite{Bai1993b}, we can assert that
\begin{align}\label{al:4}
{\rm E}m_n\left(z\right)=\frac{1-z-y_n+y_nz\delta_n+\sqrt{\left(1-z-y_n-y_nz\delta_n\right)^2-4y_nz}}{2y_nz}.
\end{align}
Comparing (\ref{al:3}) with (\ref{al:4}), it  suffices to show that
\begin{align*}
{\delta _n} \to 0.
\end{align*}

\begin{lemma}\label{lemma:2}
Under the conditions of Remark \ref{le:3}, for any $z=u+vi$ with $v>0$ and for any \ $k=1,\cdots,p$, we have
\begin{align}\label{al:8}
\left|{\rm Etr}\boldsymbol\varepsilon_k\right| \to 0.
\end{align}
\end{lemma}
\begin{proof} By Lemma \ref{lemma:7}, we have
 \begin{align*}
&\left|{\rm Etr}\boldsymbol\varepsilon_k\right|\\
=&\left| - \frac{1}{{{n^2}}}{\rm Etr}{\mathbf X}_{nk}^*{{\left(\frac{1}{n}{\mathbf X_{nk}}{\mathbf X_{nk}^*} - z{\mathbf I_{2p - 2}}\right)}^{-1}}{\mathbf X_{nk}}+2{y_n}+2 {y_n}z{\rm E}{m_n}\left(z\right)\right|\\
=&\left| - \frac{1}{n}{\rm Etr}{{\left(\frac{1}{n}{\mathbf X_{nk}}{\mathbf X_{nk}^*} - z{\mathbf I_{2p - 2}}\right)}^{-1}}\frac{1}{n}{\mathbf X_{nk}}{\mathbf X}_{nk}^* +2{y_n}+2 {y_n}z{\rm E}{m_n}\left(z\right)\right|\\
\le&\frac{2}{n}+\frac{\left|z\right|}{n}\left|{\rm E}\left[{\rm tr}{{\left(\frac{1}{n}{\mathbf X_{n}}{\mathbf X_{n}^*} - z{\mathbf I_{2p }}\right)}^{-1}}-{\rm tr}{{\left(\frac{1}{n}{\mathbf X_{nk}}{\mathbf X_{nk}^*} - z{\mathbf I_{2p - 2}}\right)}^{-1}}\right]\right|\\
\le&\frac{2}{n}+\frac{2\left|z\right|}{n\upsilon}
\to 0.\end{align*}
Then, the proof is complete.
\end{proof}
\begin{lemma}\label{lemma:6}
Under the conditions of Remark \ref{le:3}, for any $z=u+vi$ with $v>0$ and for any \ $k=1,\cdots,p$, we have
\begin{align*}
{\rm E}\left|{\rm tr}\boldsymbol\varepsilon_k^2\right| \to 0.
\end{align*}
\end{lemma}
\begin{proof} Write the form of \ $\left({\mathbf S_{n}} - z{\mathbf I_{2p}}\right)$
\[\left( {\begin{array}{*{20}{c}}
{{t_1}}&0&{{a_{12}}}&{{b_{12}}}& \cdots \\
0&{{t_1}}&{ - {{\bar b}_{12}}}&{{{\bar a}_{12}}}& \cdots \\
{{{\bar a}_{12}}}&{ - {b_{12}}}&{{t_2}}&0& \cdots \\
{{{\bar b}_{12}}}&{{a_{12}}}&0&{{t_2}}& \cdots \\
 \vdots & \vdots & \vdots & \vdots & \ddots
\end{array}} \right)\]
and denote $\mathbf R_k={\left(\frac{1}{n}{\mathbf X_{nk}}{\mathbf X_{nk}^*} - z{\mathbf I_{2p - 2}}\right)}^{-1}$. By Corollary \ref{col:1} and
\begin{align*}
&\frac{1}{n}\boldsymbol{\phi}_k^{\prime}\bar{\boldsymbol{\phi}}_k - z{\mathbf I_2} - \frac{1}{{{n^2}}}\boldsymbol{\phi}_k^{\prime}{\mathbf X}_{nk}^*{\mathbf R_k}{\mathbf X_{n k}}\bar{\boldsymbol{\phi}}_k\\
=&\left( {\begin{array}{*{20}{c}}
{\frac{1}{n}{\boldsymbol\alpha_k^{\prime}}{{\bar {\boldsymbol\alpha} }_k} - z - \frac{1}{{{n^2}}}{\boldsymbol\alpha_k^{\prime}}{\mathbf X_{nk}^*}{\mathbf R_k}{\mathbf X_{nk}}{{\bar {\boldsymbol\alpha} }_k}}&{\frac{1}{n}{\boldsymbol\alpha_k^{\prime}}{{\bar {\boldsymbol\beta} }_k} -\frac{1}{{{n^2}}}{\boldsymbol\alpha_k^{\prime}}{\mathbf X_{nk}^*}{\mathbf R_k}{\mathbf X_{nk}}{{\bar {\boldsymbol\beta} }_k}}\\
{\frac{1}{{{n}}}{\boldsymbol\beta_k^{\prime}}{{\bar {\boldsymbol\alpha} }_k}-\frac{1}{{{n^2}}}{\boldsymbol\beta_k^{\prime}}{\mathbf X_{nk}^*}{\mathbf R_k}{\mathbf X_{nk}}{{\bar {\boldsymbol\alpha} }_k}}&{\frac{1}{n}{\boldsymbol\beta_k^{\prime}}{{\bar {\boldsymbol\beta} }_k} - z - \frac{1}{{{n^2}}}{\boldsymbol\beta_k^{\prime}}{\mathbf X_{nk}^*}{\mathbf R_k}{\mathbf X_{nk}}{{\bar {\boldsymbol\beta} }_k}}
\end{array}}\right)
\end{align*}
where $\boldsymbol\alpha_k$ denotes the first column of $\boldsymbol{\phi}_k$ and $\boldsymbol\beta_k$ denotes the second column of $\boldsymbol{\phi}_k$,
we can get:
\begin{align*}
&{\frac{1}{n}{\boldsymbol\alpha_k^{\prime}}{{\bar {\boldsymbol\alpha} }_k} - z - \frac{1}{{{n^2}}}{\boldsymbol\alpha_k^{\prime}}{\mathbf X_{nk}^*}{\mathbf R_k}{\mathbf X_{nk}}{{\bar {\boldsymbol\alpha} }_k}}={\frac{1}{n}{\boldsymbol\beta_k^{\prime}}{{\bar {\boldsymbol\beta} }_k} - z - \frac{1}{{{n^2}}}{\boldsymbol\beta_k^{\prime}}{\mathbf X_{nk}^*}{\mathbf R_k}{\mathbf X_{nk}}{{\bar {\boldsymbol\beta} }_k}}
\end{align*}
and
\begin{align*}
{\frac{1}{n}{\boldsymbol\alpha_k^{\prime}}{{\bar {\boldsymbol\beta} }_k} -\frac{1}{{{n^2}}}{\boldsymbol\alpha_k^{\prime}}{\mathbf X_{nk}^*}{\mathbf R_k}{\mathbf X_{nk}}{{\bar {\boldsymbol\beta} }_k}}
={\frac{1}{{{n}}}{\boldsymbol\beta_k^{\prime}}{{\bar {\boldsymbol\alpha} }_k}-\frac{1}{{{n^2}}}{\boldsymbol\beta_k^{\prime}}{\mathbf X_{nk}^*}{\mathbf R_k}{\mathbf X_{nk}}{{\bar {\boldsymbol\alpha} }_k}}
=0.
\end{align*}
Note that
\begin{align*}
&\frac{1}{n}\boldsymbol{\phi}_k^{\prime}\bar{\boldsymbol{\phi}}_k - z{\mathbf I_2} - \frac{1}{{{n^2}}}\boldsymbol{\phi}_k^{\prime}{\mathbf X}_{nk}^*{\mathbf R_k}{\mathbf X_{n k}}\bar{\boldsymbol{\phi}}_k
=\boldsymbol\varepsilon_k+ \left(1 - z - {y_n} - {y_n}z{\rm E}{m_n}(z)\right){\mathbf I_2},
\end{align*}
thus we have
\begin{align*}
\boldsymbol\varepsilon_k=\left( {\begin{array}{*{20}{c}}
\theta &0\\
0&\theta
\end{array}}\right)
\end{align*}
where
\begin{align*}
\theta=\frac{1}{n}{\boldsymbol\alpha_k^{\prime}}{{\bar {\boldsymbol\alpha} }_k} - z - \frac{1}{{{n^2}}}{\boldsymbol\alpha_k^{\prime}}{\mathbf X_{nk}^*}{\mathbf R_k}{\mathbf X_{nk}}{{\bar {\boldsymbol\alpha} }_k}-\left(1 - z - {y_n} - {y_n}z{\rm E}{m_n}\left(z\right)\right).
\end{align*}
Let $\widetilde{\rm E}\left(\cdot\right)$ denote the conditional expectation given $\left\{\mathbf x_j,j=1,\cdots,n;j\neq k\right\}$, then we get
\begin{align}\label{al:5}
&{\rm E}\left|{\rm tr}\boldsymbol\varepsilon_k^2\right|=2{\rm E}\left|\theta\right|^2=\frac{1}{2}{\rm E}\left|{\rm tr}\boldsymbol\varepsilon_k\right|^2\notag\\
\leq&{2}\left[{\rm E}\left|{\rm tr}\boldsymbol\varepsilon_k-\widetilde{\rm E}{\rm tr}\boldsymbol\varepsilon_k\right|^2+{\rm E}\left|\widetilde{\rm E}{\rm tr}\boldsymbol\varepsilon_k-{\rm E}{\rm tr}\boldsymbol\varepsilon_k\right|^2+\left|{\rm E}{\rm tr}\boldsymbol\varepsilon_k\right|^2\right].
\end{align}
We will estimate the inequality above by the following three steps.

\textbf{(1):} We start from the bound of first term of (\ref{al:5}).\\
Write $\mathbf T=\left(t_{jl}\right)=\mathbf I_{2n}-\frac{1}{n}{\mathbf X}_{nk}^*{\mathbf R_k}{\mathbf X_{nk}}$, where $t_{jl}=\left(\begin{array}{cc}e_{jl}&f_{jl}\\h_{jl}&g_{jl}\end{array}\right)$, then we have
\begin{align*}
&{\rm tr}\boldsymbol\varepsilon_k-\widetilde{\rm E}{\rm tr}\boldsymbol\varepsilon_k\\
=&{\rm tr}\left(\frac{1}{n}\boldsymbol{\phi}_k^{\prime}\bar{\boldsymbol{\phi}}_k  - \frac{1}{{{n^2}}}\boldsymbol{\phi}_k^{\prime}{\mathbf X}_{nk}^*{\mathbf R_k}{\mathbf X_{nk}}\bar{\boldsymbol{\phi}}_k\right)-{\rm tr}\left(\mathbf I_2 - \frac{1}{{{n^2}}}{\mathbf X}_{nk}^*{\mathbf R_k}{\mathbf X_{nk}}\right)\\
=&\frac{1}{n}{\rm tr}\left(\boldsymbol{\phi}_k^{\prime}\mathbf T\bar{\boldsymbol{\phi}}_k-\mathbf T\right)\\
=&\frac{1}{n}\left(\sum_{j=1}^{n}{\rm tr}\left(\left\|x_{kj}\right\|^2-1\right)t_{jj}+\sum_{j\neq l}^{}{\rm tr}\left(x_{kl}^*x_{kj}t_{jl}\right)\right).
\end{align*}
By elementary calculation, we obtain
\begin{align}\label{al:7}
&\widetilde{\rm E}\left|{\rm tr}\boldsymbol\varepsilon_k-\widetilde{\rm E}{\rm tr}\boldsymbol\varepsilon_k\right|^2\notag\\
=&\frac{1}{n^2}\bigg(\sum_{j=1}^{n}\widetilde{\rm E}\left|{\rm tr}\left(\left\|x_{kj}\right\|^2-1\right)t_{jj}\right|^2+\sum_{j\neq l}^{}\widetilde{\rm E}\big[{\rm tr}\left(x_{kl}^*x_{kj}t_{jl}\right){\rm tr}\left(x_{kj}^*x_{kl}t_{jl}^*\right)\notag\\
&\qquad+{\rm tr}\left(x_{kl}^*x_{kj}t_{jl}\right){\rm tr}\left(x_{kl}^*x_{kj}t_{lj}^*\right)\big]\bigg)\notag\\
\le&\frac{1}{n^2}\left(\sum_{j=1}^{n}\widetilde{\rm E}\left(\left\|x_{kj}\right\|^2-1\right)^2\left|e_{jj}+g_{jj}\right|^2+2\sum_{j\neq l}\widetilde{\rm E}\left|{\rm tr}\left(x_{kl}^*x_{kj}t_{jl}\right)\right|^2\right)\notag\\
\le&\frac{C}{n^2}\left(\eta_n^2n\sum_{j=1}^{n}\left(\left|e_{jj}\right|^2
+\left|g_{jj}\right|^2\right)+\sum_{j\neq l}^{}\left(\left|e_{jl}\right|^2+\left|f_{jl}\right|^2+\left|g_{jl}\right|^2+\left|h_{jl}\right|^2\right)\right)\notag\\
\le&\frac{C\eta_n^2}{n}\sum_{j=1}^{n}\left(\left|e_{jj}\right|^2
+\left|g_{jj}\right|^2\right)+\frac{C}{n^2}\sum_{j, l}^{}\left(\left|e_{jl}\right|^2+\left|f_{jl}\right|^2+\left|g_{jl}\right|^2+\left|h_{jl}\right|^2\right)\notag\\
\le&\frac{C\eta_n^2}{n}{\rm tr} \mathbf T \mathbf T^*+\frac{C}{n^2}{\rm tr} \mathbf T \mathbf T^*.
\end{align}
To complete the estimation, we only need to show that  ${\rm tr} \mathbf T \mathbf T^*$ is a bounded random variable. For \ $\frac{1}{\sqrt n}{\mathbf X}_{nk}$, there exist \ $(2p-2)\times q$  orthonormal matrix\ $\mathbf U$ and \ $2n\times q$ orthonormal matrix \ $\mathbf V$ such that
\begin{align*}
\frac{1}{\sqrt n}{\mathbf X}_{nk}=\mathbf U {\rm diag}\left(s_1,\cdots,s_q\right)\mathbf V^*
\end{align*}
where \ $s_1,\cdots,s_q$ are the singular values of \ $\frac{1}{\sqrt n}{\mathbf X}_{nk}$ and $q=(2p-2)\land 2n$. Then, we get
\begin{align*}
&\mathbf I_{2n}-\mathbf T=\left(\frac{1}{\sqrt n}{\mathbf X}_{nk}^*\right){\mathbf R_k}\left(\frac{1}{\sqrt n}{\mathbf X}_{nk}\right)\\
=&\mathbf V {\rm diag}\left(s_1,\cdots,s_q\right)\mathbf U^*\mathbf U {\rm diag}\left(s_1^2-z,\cdots,s_q^2-z\right)^{-1}\mathbf U^*
\mathbf U {\rm diag}\left(s_1,\cdots,s_q\right)\mathbf V^*\\
=&\mathbf V {\rm diag}\left(\frac{s_1^2}{s_1^2-z},\cdots,\frac{s_q^2}{s_q^2-z}\right)\mathbf V^*.
\end{align*}
Thus
\begin{align*}
\mathbf T=\mathbf I_{2n}-\mathbf V {\rm diag}\left(\frac{s_1^2}{s_1^2-z},\cdots,\frac{s_q^2}{s_q^2-z}\right)\mathbf V^*
=\mathbf V {\rm diag} \left(\frac{-z}{s_1^2-z},\cdots,\frac{-z}{s_q^2-z}\right)\mathbf V^*
\end{align*}
which implies that
\begin{align}\label{al:6}
{\rm tr} \mathbf T \mathbf T^*=\sum_{j=1}^{q}\frac{\left|z\right|^2}{\left|s_j^2-z\right|^2}\le \frac{2n|z|^2}{\upsilon^2}.
\end{align}
By (\ref{al:7}) and (\ref{al:6}),
we obtain
\begin{align}\label{al:9}
{\rm E}\left|{\rm tr}\boldsymbol\varepsilon_k-\widetilde{\rm E}{\rm tr}\boldsymbol\varepsilon_k\right|^2\to 0.
\end{align}

\textbf{(2):} Next, we bound the second term of (\ref{al:5}). Note that
\begin{align*}
\widetilde{\rm E}{\rm tr}\boldsymbol\varepsilon_k-{\rm E}{\rm tr}\boldsymbol\varepsilon_k=\frac{z}{n}\left({\rm Etr}{\mathbf R_k}-{\rm tr}{\mathbf R_k}\right).
\end{align*}
Using the martingale decomposition method (similar to the proof of (\ref{eq:6})), we have
\begin{align}\label{al:10}
{\rm E}\left|\widetilde{\rm E}{\rm tr}\boldsymbol\varepsilon_k-{\rm E}{\rm tr}\boldsymbol\varepsilon_k\right|^2=&\frac{\left|z\right|^2}{n^2}{\rm E}\left|{\rm Etr}{\mathbf R_k}-{\rm tr}{\mathbf R_k}\right|^2\le\frac{4\left|z\right|^2}{n\upsilon^2}\to 0.
\end{align}

\textbf{(3):} Finally, combining  (\ref{al:8}), (\ref{al:5}), (\ref{al:9}) and (\ref{al:10}), we conclude that
\begin{align*}
{\rm E}\left|{\rm tr}\boldsymbol\varepsilon_k^2\right|\to 0,
\end{align*}
which complete the proof of the lemma.
\end{proof}
\begin{lemma}
For any $z=u+vi$ with $v>0$, we have
\begin{align*}
\delta_n(z)\to 0.
\end{align*}
\end{lemma}
\begin{proof} By (\ref{al:11}), we can write
\begin{small}
\begin{align*}
{\delta _n} = & - \frac{1}{2p\left({{1 - z - {y_n} - {y_n}z{\rm E}{m_n}\left(z\right)}}\right)^2}\sum\limits_{k = 1}^p{\rm Etr}{\boldsymbol\varepsilon _k}\\
+ &\frac{1}{2p\left({{1 - z - {y_n} - {y_n}z{\rm E}{m_n}\left(z\right)}}\right)^2}{\sum\limits_{k = 1}^p{\rm Etr}\left\{{\boldsymbol\varepsilon _k^2}{\left(\left(1 - z - {y_n} - {y_n}z{\rm E}{m_n}\left(z\right)\right){\mathbf I_2} + {\boldsymbol\varepsilon _k}\right)^{ - 1}}\right\}}.
\end{align*}
\end{small}
Note that $$\Im \left(1 - z - {y_n} - {y_n}z{\rm E}{m_n}\left(z\right)\right) <  - \upsilon,$$ which implies that
\begin{align}\label{al:12}
\left| {1 - z - {y_n} - {y_n}z{\rm E}{m_n}\left(z\right)} \right| > \upsilon.
\end{align}
By Lemma \ref{lemma:2} and (\ref{al:12}), we have
\begin{align}\label{al:14}
\left|\frac{1}{2p\left({{1 - z - {y_n} - {y_n}z{\rm E}{m_n}\left(z\right)}}\right)^2}\sum\limits_{k = 1}^p{\rm Etr}{\boldsymbol\varepsilon _k}\right|
\le\frac{1}{2p\upsilon^2}\sum\limits_{k = 1}^p\left|{\rm Etr}{\boldsymbol\varepsilon _k}\right|
\to0.
\end{align}
Rewrite
\begin{align*}
&\boldsymbol\varepsilon_k+ \left(1 - z - {y_n} - {y_n}z{\rm E}{m_n}\left(z\right)\right){\mathbf I_2}\\=&\left( {\begin{array}{*{20}{c}}
{\frac{1}{n}{\boldsymbol\alpha_k^{\prime}}{{\bar {\boldsymbol\alpha} }_k} - z - \frac{1}{{{n^2}}}{\boldsymbol\alpha_k^{\prime}}{\mathbf X_{nk}^*}{\mathbf R_k}{\mathbf X_{nk}}{{\bar {\boldsymbol\alpha} }_k}}&0\\
0&{\frac{1}{n}{\boldsymbol\beta_k^{\prime}}{{\bar {\boldsymbol\beta} }_k} - z - \frac{1}{{{n^2}}}{\boldsymbol\beta_k^{\prime}}{\mathbf X_{nk}^*}{\mathbf R_k}{\mathbf X_{nk}}{{\bar {\boldsymbol\beta} }_k}}
\end{array}}\right)\\+ &\left(1 - z - {y_n} - {y_n}z{\rm E}{m_n}\left(z\right)\right){\mathbf I_2}.
\end{align*}
Note that  $$\begin{array}{l}
\Im \left({\frac{1}{n}{\boldsymbol\alpha_k^{\prime}}{{\bar {\boldsymbol\alpha} }_k} - z - \frac{1}{{{n^2}}}{\boldsymbol\alpha_k^{\prime}}{\mathbf X_{nk}^*}{{\left(\frac{1}{n}{\mathbf X_{nk}}{\mathbf X_{nk}^*} - z{\mathbf I_{2p - 2}}\right)}^{ - 1}}{\mathbf X_{nk}}{{\bar {\boldsymbol\alpha} }_k}}\right)\\
 =  -\upsilon\left( {1 + {\boldsymbol\alpha_k^{\prime}}{\mathbf X_{nk}^*}{{\left[ {{{\left(\frac{1}{n}{\mathbf X_{nk}}{\mathbf X_{nk}^*} - z{\mathbf I_{2p - 2}}\right)}^2} + {\upsilon^2}{\mathbf I_{2p - 2}}} \right]}^{ - 1}}{\mathbf X_{nk}}{{\bar {\boldsymbol\alpha} }_k}} \right) <  - \upsilon\end{array}$$
which implies that
\begin{align}\label{al:13}
\left|\left({\frac{1}{n}{\boldsymbol\alpha_k^{\prime}}{{\bar {\boldsymbol\alpha} }_k} - z - \frac{1}{{{n^2}}}{\boldsymbol\alpha_k^{\prime}}{\mathbf X_{nk}^*}{{\left(\frac{1}{n}{\mathbf X_{nk}}{\mathbf X_{nk}^*} - z{\mathbf I_{2p - 2}}\right)}^{ - 1}}{\mathbf X_{nk}}{{\bar {\boldsymbol\alpha} }_k}}\right) \right| > \upsilon.
\end{align}
By Lemma \ref{lemma:6}, (\ref{al:12}) and (\ref{al:13}), we have
\begin{small}
\begin{align}\label{al:15}
&\left|\frac{1}{2p\left({{1 - z - {y_n} - {y_n}z{\rm E}{m_n}\left(z\right)}}\right)^2}{\sum\limits_{k = 1}^p{\rm Etr}\left\{{\boldsymbol\varepsilon _k^2}{\left(\left(1 - z - {y_n} - {y_n}z{\rm E}{m_n}\left(z\right)\right){\mathbf I_2} + {\boldsymbol\varepsilon _k}\right)^{ - 1}}\right\}}\right|\notag\\
\le&\frac{1}{2p\upsilon^3}\sum\limits_{k = 1}^p{\rm E}\left|{\rm tr}{\boldsymbol\varepsilon _k^2}\right|
\to0.
\end{align}
\end{small}
Combining (\ref{al:14}) and (\ref{al:15}), we get
\begin{small}
\begin{align*}
\left|{\delta _n}\right| \le& \left| \frac{1}{2p\left({{1 - z - {y_n} - {y_n}z{\rm E}{m_n}\left(z\right)}}\right)^2}\sum\limits_{k = 1}^p{\rm Etr}{\boldsymbol\varepsilon _k}\right|\\
+ &\left|\frac{1}{2p\left({{1 - z - {y_n} - {y_n}z{\rm E}{m_n}\left(z\right)}}\right)^2}{\sum\limits_{k = 1}^p{\rm Etr}\left\{{\boldsymbol\varepsilon _k^2}{\left(\left(1 - z - {y_n} - {y_n}z{\rm E}{m_n}\left(z\right)\right){\mathbf I_2} + {\boldsymbol\varepsilon _k}\right)^{ - 1}}\right\}}\right|\\
\to &0.
\end{align*}
\end{small}
Then the proof of this lemma is complete.
\end{proof}
Now, we have completed the proof of the mean convergence
$${\rm E}{m_n}\left(z\right) \to m\left(z\right).$$

\subsubsection{ Completion of the proof of Theorem \ref{th:1}}
We need the last part of Chapter 2 of \cite{bai2010spectral} for completing the proof Theorem \ref{th:1}. For the readers convenience we repeat here. By \ Section \ref{se:1} and \ Section \ref{se:2}, for any fixed $z\in \mathbb C^+$, we have
\begin{align*}
m_n\left(z\right)\to m\left(z\right),\mbox{a.s..}
\end{align*}
That is, for each\  $z\in \mathbb C^+$, there exists a null set $N_z$ (i.e.,$\mbox{P}\left(N_z\right)=0$) such that
$$m_n\left(z,w\right)\to m\left(z\right), \  \mbox{for all}\  w\in N_z^c.$$
Now, let \ $\mathbb C_0^+$ be a dense subset of \ $\mathbb C^+$ (e.g., all $z$ of rational real and imaginary parts) and let $N=\bigcup_{z\in \mathbb C_0^+} N_{z}$. Then
$$m_n\left(z,w\right)\to m\left(z\right), \ \mbox{for all } w\in N^c \mbox{and} \ z\in\mathbb C_0^+.$$
Let \ $\mathbb C_m^+=\left\{z\in\mathbb C^+,\ \Im z >1/m,\ |z|\le m\right\}$. When \ $z\in\mathbb C_m^+$, we have $|s_n\left(z\right)|\le m$. Applying Lemma  \ref{lemma:12}, we have
$$m_n\left(z,w\right)\to m\left(z\right), \ \mbox{for all } \ w\in N^c \ \mbox{and} \ z\in \mathbb C_m^+.$$
Since the convergence above holds for every \ $m$, we conclude that
$$m_n\left(z,w\right)\to m\left(z\right), \ \mbox{for all } \ w\in N^c \ \mbox{and} \ z\in \mathbb C^+.$$
Applying Lemma \ref{lemma:13}, we conclude that
$$F^{\mathbf S_n}\overset{w}{\rightarrow} F, \ \mbox{a.s..}$$

\section{Appendix}
We introduce some results that will be used in this paper.
\begin{lemma}[Corollary A.42 in \cite{bai2010spectral}]\label{lemma:4}
Let $A$ and $B$ be two $p \times n$  matrices and denote the ESD of $S = A{A^ * }$ and $\widetilde S = B{B^ * }$   by ${F^S}$ and ${F^{\widetilde S}}$.  Then,
$${L^4}\left({F^S},{F^{\widetilde S}}\right) \le \frac{2}{{{p^2}}}\left(tr\left(A{A^ * } + B{B^ * }\right)\right)\left(tr\left[\left(A - B\right){\left(A - B\right)^ * }\right]\right),$$
 where $L\left(\cdot,\cdot \right)$  denotes the L\'{e}vy distance.
\end{lemma}

\begin{lemma}[Theorem A.44 in \cite{bai2010spectral}]\label{lemma:3}
Let $A$ and $B$ be $p \times n$ complex matrices. Then, $$\left\| {{F^{A{A^ * }}} - {F^{B{B^ * }}}} \right\| \le \frac{1}{p}rank\left(A - B\right).$$
\end{lemma}

\begin{lemma}[Bernstein's inequality]\label{lemma:10}
If \ $\tau_1,\cdots,\tau_n$ are independent random variables with mean zero and uniformly bounded by $b$, then, for any $\varepsilon > 0$,
\begin{displaymath}
{\rm P}\left(\left|\sum_{j=1}^{n}\tau_j\right|\ge \varepsilon\right)\le 2exp\left(-\varepsilon^2/\left[2\left(B_n^2+b\varepsilon\right)\right]\right)
\end{displaymath}
where $ B_n^2={\rm E}\left(\tau_1+\cdots+\tau_n\right)^2.$
\end{lemma}

\begin{lemma}[see $\left(A.1.12\right)$ in \cite{bai2010spectral}]\label{lemma:7}
Let $z = u + iv, v > 0, $ and let $A$ be an $n \times n$ Hermitian matrix.  ${A_k}$ be the k-th major sub-matrix of $A$ of order $(n-1)$, to be the matrix resulting from the $k$-th row and column from $A$.  Then
$$\left| {{\rm tr}{{\left(A - z{I_n}\right)}^{ - 1}} -{ \rm tr}{{\left({A_k} - z{I_{n - 1}}\right)}^{ - 1}}} \right| \le \frac{1}{\upsilon}.$$
\end{lemma}

\begin{lemma}[Theorem A.4 in \cite{bai2010spectral}]\label{lemma:9}
For an $n\times n$ Hermitian $A$,define $A_k$,called a major submatrix of order $n-1$,to be the matrix resulting from the $k$-th row and column from $A$.
If both $A$ and $A_k$,$k=1,\cdots,n$, are nonsingular,and if we write $A^{-1}=[a^{kl}]$,then
\begin{displaymath}
a^{kk}=\frac{1}{a_{kk}-\alpha_k^*A_k^{-1}\beta_k},
\end{displaymath}
and hence
\begin{displaymath}
{\rm tr}\left(A^{-1}\right)=\sum_{k=1}^{n}\frac{1}{a_{kk}-\alpha_k^*A_k^{-1}\beta_k},
\end{displaymath}
\end{lemma}

\begin{lemma}[Burkholder's inequality ]\label{lemma:8}
Let $\{ {X_k}\} $ be a complex martingale difference sequence with respect to the increasing $\sigma$-field. Then, for $p > 1,$ ${\rm E}{\left| {\sum {{X_k}} } \right|^p} \le {K_p}{\rm E}{\left({\sum {\left| {{X_k}} \right|} ^2}\right)^{p/2}}$.
\end{lemma}

\begin{lemma}[Rosenthal's inequality ]\label{lemma:5}
Let $ X_i $ are independent with zero means,  then we have, for some constant $C_k$:
$${\rm E}\left|\sum X_i\right|^{2k} \leq C_k\left(\sum{\rm E}\left|X_i\right|^{2k}+\left(\sum {\rm E}\left|X_i\right|^2\right)^k\right). $$
\end{lemma}

\begin{lemma}[Lemma 2.14 in \cite{bai2010spectral}]\label{lemma:12}
Let $f_1,f_2,\cdots$ be analytic in $D$, a connected open set of $\mathbb C$, satisfying $\left|f_n\left(z\right)\right|\le M$ for every $n$ and $z$ in $D$, and $f_n\left(z\right)$ converges as $n \to \infty$ for each $z$ in a subset of $D$ having a limit point in $D$. Then there exists a function $f$ analytic in $D$ for which $f_n\left(z\right) \to f\left(z\right)$ and $f_n^{\prime} \to f^{\prime}\left(z\right)$ for all $z\in D$. Moreover, on any set bounded by a contour interior to $D$, the convergence is uniform and $\left\{f_n^{\prime}\left(z\right)\right\}$ is uniformly bounded.
\end{lemma}

\begin{lemma}[Theorem B.9 in \cite{bai2010spectral}]\label{lemma:13}
Assume that $\left\{G_n\right\}$ is a sequence of functions of bounded variation and $G_n\left(-\infty\right)=0$ for all $n$. Then,
$$\lim_{n\to \infty}m_{G_n}\left(z\right)=m\left(z\right)\ \forall z\in D$$
if and only if there is a function of bounded variation $G$ with $G\left(-\infty\right)=0$  and Stieltjes transform $m\left(z\right)$ and such that $G_n\to G$ vaguely.
\end{lemma}

\end{document}